\documentclass[12pt]{article}
\usepackage{color,bbold}
\usepackage{amssymb,amsmath,amsthm,graphicx, color,amsfonts,tabularx}
\usepackage[english]{babel}
\usepackage{afterpage}
\usepackage{authblk, relsize}
\usepackage{float}
\usepackage{geometry}
\def\supp{\mathop{\rm supp}\nolimits}
\def\lip{\mathop{\rm Lip}\nolimits}
\newtheorem{theorem}{Theorem}[section]

\newtheorem{lemma}[theorem]{Lemma}
\newtheorem{proposition}[theorem]{Proposition}
\newtheorem{corollary}[theorem]{Corollary}
\newtheorem{definition}[theorem]{Definition}

\renewenvironment{proof}[1][.]{%
\bigskip\noindent{\bf Proof#1 }}{%
\hfill$\blacksquare$\bigskip}

\begin{document}
\pagestyle{myheadings}
\title{Large deviation for Gibbs probabilities at zero temperature and invariant idempotent probabilities for iterated function systems  }
\author[1]{Jairo K. Mengue
\thanks{E-mail: jairo.mengue@ufrgs.br}}
\author[2]{Elismar R. Oliveira
\thanks{E-mail: elismar.oliveira@ufrgs.br}}
\affil[1,2]{Universidade Federal do Rio Grande do Sul}

\date{\today}
\maketitle

\begin{abstract}
We consider two compact metric spaces $J$ and $X$ and a uniform contractible iterated function system $\{\phi_j: X \to X \, | \, j \in J \}$. For a Lipschitz continuous function $A$ on $J \times X$  and for each $\beta>0$ we consider the Gibbs probability $\rho_{_{\beta A}}$. Our goal is to study a large deviation principle for such family of probabilities as $\beta \to +\infty$ and its connections with idempotent probabilities. In the non-place dependent case ($A(j,x)=A_j,\,\forall x\in X$) we will prove that $(\rho_{_{\beta A}})$ satisfy a LDP and $-I$ (where $I$ is the deviation function) is the density of the unique invariant idempotent probability for a mpIFS associated to $A$. In the place dependent case, we prove that, if $(\rho_{_{\beta A}})$ satisfy a LDP, then $-I$ is the density of an invariant idempotent probability. Such idempotent probabilities were recently characterized through the Ma\~{n}\'{e} potential and Aubry set, therefore  we will obtain an identical characterization for $-I$.
\end{abstract}
\vspace {.8cm}




\section{Introduction}

We consider two compact metric spaces $(X,d_X)$ and $(J,d_J)$ and denote by $C(X,\mathbb{R})$ and $C(J\times X,\mathbb{R})$, the set of continuous functions from $X$ to $\mathbb{R}$ and $J\times X$ to $\mathbb{R}$, respectively.

\begin{definition}\label{def:ucIFS}
	Let $(X,d_X)$ and $(J,d_J)$ be compact metric spaces. 	A \emph{uniformly contractible iterated function system}  $(X,(\phi_j)_{j\in J})$ is a family of maps $\{\phi_j:X\to X\,|\,j\in J\}$  satisfying:  there exists $0<\gamma<1$ such that
	\begin{equation}\label{eq: gamma contraction}
		d_X(\phi_{j_1}(x_1),\phi_{j_2}(x_2)) \leq \gamma\cdot[d_J(j_1,j_2)+ d_X(x_1,x_2)],\,\,\forall j_1,j_2\in J, \forall x_1,\,x_2\in X.
	\end{equation}
We also consider the continuous map $\phi:J\times X\to X,$ defined by $\phi(j,x) =\phi_j(x)$.
\end{definition}

By one hand, in \cite{MO} it is described a thermodynamic formalism for uniformly contractible IFS $(X,(\phi_j)_{j\in J})$. On the other hand, in \cite{MO2} it is characterized the invariant idempotent probabilities for the so called max-plus IFS. In the present work we exhibit a connection between these two settings. Our objective is to study a large deviation principle (LDP) for Gibbs probabilities in thermodynamic formalism for IFS, when the temperature goes to zero, as an application of results in \cite{MO2}. As we will see, the \textit{Ma\~{n}\'{e} potential} and \textit{Aubry set} will play an important role in the description of the deviation function.

From now on we always consider the Borel sigma algebra in each metric space. Let $B: J\times X  \to (0,+\infty)$ be a continuous function and $\nu$ be a fixed Borel probability on $J$ satisfying $\supp(\nu) = J$. We define the transfer operator $\mathcal{L}_{\phi,B,\nu}:C(X,\mathbb{R}) \to C(X,\mathbb{R})$ by
\[\mathcal{L}_{\phi,B,\nu}(f)(x) = \int_J B(j,x)\cdot f(\phi_j(x))\,d\nu(j),\]
for any $f \in C(X,\mathbb{R})$.

For each $\beta >0$ (which is interpreted as the inverse of temperature in thermodynamic formalism) and a Lipschitz continuous function $A:J\times X\to\mathbb{R}$ we also consider the operator $\mathcal{L}_{\phi,e^{\beta A},\nu}$.
It admits a unique pair $(\lambda_{\beta A}, h_{\beta A})$ where $\lambda_{\beta A}$ is a positive number and $h_{\beta A}:X\to\mathbb{R}$ is a positive and continuous function, such that $\sup_{x\in X}h_{\beta A}(x) = 1$ and  $\mathcal{L}_{\phi,e^{\beta A},\nu}(h_{\beta A})=\lambda_{\beta A}\cdot h_{\beta A}$.
The function $q^{\beta}:J\times X \to \mathbb{R}$, defined by $q^{\beta}(j,x):=\frac{e^{\beta A(j,x)}\cdot h_{\beta A}(\phi_j(x))}{ \lambda_{\beta A}\cdot h_{\beta A}(x)}$, satisfies $\mathcal{L}_{\phi,q^{\beta},\nu}(1)=1$
and the dual operator of $\mathcal{L}_{\phi,q^{\beta},\nu}$ has a unique invariant probability $\rho_{_{\beta A}}$ on $X$, that is,
\begin{equation}\label{eq: rho invariant} \int_X f(x)d\rho_{_{\beta A}}(x) =  \int_X \int_J {q^{\beta}(j,x)}\cdot f(\phi_j(x))\,d\nu(j)\,d\rho_{_{\beta A}}(x), \,\, \forall f\in C(X,\mathbb{R}).
\end{equation}
We call $\rho_{_{\beta A}}$ a \textit{Gibbs probability} on $X$. Our objective is to consider a large deviation principle for $(\rho_{_{\beta A}})$ as $\beta \to +\infty$ and to present its connections with invariant idempotent probabilities.

The Ruelle (transfer) operator was introduced by D. Ruelle \cite{R} in the context of equilibrium states - Gibbs measures - of infinite one-dimensional lattice gas. It  was extensively studied for IFS in the last decades and a remarkable result was Fan's Theorem (see \cite{FanRuelle}, 	Theorem 1.1) which shows the existence of a positive eigenfunction and a Gibbs probability for a contractive IFS with Dini continuous place dependent positive weights. Several extensions of this result appeared studying the Ruelle operator, its rate of convergence and the existence of eigenfunctions and eigenmeasures.  In \cite{LY01}, for example, it is studied the conditions for the existence of such eigenfunctions together with the Perron-Frobenius property for non-expansive finite IFSs. In \cite{Ye03} it is studied the speed of convergence of the Ruelle operator for a finite IFS, where the weights are positive Lipschitz continuous and the system satisfies an average contractive condition.  In \cite{Ye13}, it is studied the existence of eigenfunctions and eigenmeasures of the Ruelle operator for a finite weakly contractive IFS.

A key generalization was given by Stenflo~\cite{Ste}, where random iterations are used to represent the iterations of a so called IFS with probabilities, for an arbitrary measurable set of maps. Their approach is slightly different from the previous works on IFS with probabilities. Its main goal was to establish, when the transfer operator is Feller, that there exists a unique attracting invariant measure for the Markov chain generated by the IFS. It is worth to mention that this approach was previously introduced by \cite{BDEG}, for an average-contractive system, to show the existence of invariant measures. In \cite{BDEG}, it is assumed that $J$ is finite and the maps are Lipschitz continuous, with place dependent and Dini continuous weights.

 One can also introduce a thermodynamic formalism based on holonomic measures according to \cite{LO}, \cite{CO17}, \cite{MO} and \cite{Brasil2022}, where several results are obtained for fairly general IFSs.

There is a special class of IFS where the thermodynamic formalism is well known: if $J:=\{1,...,d\}$ is a finite set,  $X:=J^{\mathbb{N}}$ and $\phi_j(x)=\phi_j(x_1,x_2,x_3,...) = (j,x_1,x_2,x_3,...)$ where $x=(x_1,x_2,x_3,...)\in X$, then $(\phi_j)_{j\in J}$ are the inverse branches of the shift map $\sigma$ acting in $J^{\mathbb{N}}$. We call this model by full shift for finite alphabet. In this case, the thermodynamic formalism for zero temperature is also widely studied as well as its connection with ergodic optimization and max-plus algebra \cite{BLL}. Considering this case, results concerning large deviations, when the temperature goes to zero, can be founded in \cite{BLT} and \cite{M}.

In the setting of full shift for finite alphabet it is proved in \cite{BLT} (more generally for sub-shifts of finite type) that $(\rho_{_{\beta A}})$ satisfies a LDP, since $A$ has a unique maximizing measure. In this case the deviation function $I$ evaluated at any point $x\in X$, is given as by  $I(x) = R(x)+R(\sigma(x))+R(\sigma^2(x))+\ldots$, where the function $R$ is also presented. This is problematic for a general IFS, because we have at hands just the maps $\phi_j$ which in the full shift correspond to the inverse branches of $\sigma$.

In \cite{M} the main idea is to use Prop. 3.2 of \cite{PP}. We believe that such strategy can not be replicated from the full shift for finite alphabet to a general IFS. In \cite{M} it is also proved that the deviation function presented in \cite{BLT} fails without the hypothesis of uniqueness of maximizing measure and another characterization of $I$ is presented. The characterization of the deviation function for IFS in the present work is similar to one which appears in \cite{M}.

In order to study LDP for Gibbs probabilities in a uniformly contractible IFS we will apply recent results from \cite{MO2}. We will consider the function $q=\lim_{\beta\to\infty}q^{\beta}$, when it exists. In the non-place dependent case ($A(j,x)=A_j,\,\forall x\in X$) is easy to show that $q$ exists. In this case, it is proved in \cite{MO2} that there exists a unique invariant idempotent probability for the mpIFS $(X,\phi,q)$ and it is presented its density $\lambda$. In this case we will prove that $(\rho_{_{\beta A}})$ satisfy a LDP with deviation function $I=-\lambda$. In the place dependent case, the situation is more complicated. As far as we know it is not proved a LDP or the existence of the limit $q$, even in the full shift for finite alphabet. In the present work we prove that, if $(\rho_{_{\beta A}})$ satisfy a LDP and the limit $q$ there exists, then $I=-\lambda$ where $\lambda$ is the density of an invariant idempotent probability for the mpIFS $(X,\phi,q)$. As such idempotent probabilities were recently characterized through the Ma\~{n}\'{e} potential and Aubry set in \cite{MO2}, we will obtain a characterization for $I$ using the Ma\~{n}\'{e} potential and Aubry set too.

We would like to point out that despite the fact that the theory of IFS with compact space of maps is fairly  old (see \cite{Lewe93}, \cite{mendivil1998generalization}, among others), only few works are known for such more general setting. Furthermore, as far as we know, the LDP theory for IFS, at zero temperature, is absent in the literature.  Although, the thermodynamic formalism and multifractal analysis of Gibbs probabilities has shown notable advances in the last few years with many generalizations (for instance, \cite{HMU02}, \cite{JP07}, \cite{Mih22} and \cite{Brasil2022}, among others) thus it seems to be of paramount importance to find further connections and developments in this areas.

We propose, initially, in Section~\ref{sec:Thermodynamic}, to present a discussion concerning thermodynamic formalism for IFS, following mainly \cite{MO}. In the sequence, we extend to the setting of IFS some well known results, in thermodynamic formalism, for the zero temperature case. Right away, in Section~\ref{sec: idempotent} we present the main definitions and results contained in \cite{MO2}, concerning invariant idempotent probabilities for max-plus IFS. After the introduction of these elements we then present their connection by considering large deviation principles in Section~\ref{sec:Large Deviation}.

\section{Thermodynamic Formalism for IFS}\label{sec:Thermodynamic}

In this section we describe the setting presented mainly in \cite{MO}, concerning the thermodynamic formalism for IFS. The characterization of the entropy and its relation with the Kullback-Leibler divergence is presented in \cite{LopM}.

For each $\beta >0$ and a Lipschitz continuous function $A:J\times X\to\mathbb{R}$ we consider the operator
\[\mathcal{L}_{\phi,e^{\beta A},\nu}(f)(x) = \int_J e^{\beta A(j,x)}\cdot f(\phi_j(x))\,d\nu(j),\]
where $\nu$ is a fixed  probability on $J$ with $\supp(\nu)=J$.

\begin{proposition}\label{prop: eigen}
Suppose that $A:J\times X\to\mathbb{R}$ has Lipschitz constant $\lip(A)$. Then $\mathcal{L}_{\phi,e^{\beta A},\nu}$ admits a unique pair $(\lambda_{\beta A}, h_{\beta A})$ where $\lambda_{\beta A}$ is a positive number and $h_{\beta A}:X\to\mathbb{R}$ is a positive and continuous function, such that $\sup_{x\in X}h_{\beta A}(x) = 1$ and  $\mathcal{L}_{\phi,e^{\beta A},\nu}(h_{\beta A})=\lambda_{\beta A}\cdot h_{\beta A}$. Furthermore, $h_{\beta A}$ is Lipschitz continuous and $\lip(h_{\beta A})\leq e^{\beta \cdot \frac{\lip(A)}{1-\gamma}}$, where $\gamma$ satisfies \eqref{eq: gamma contraction}.
\end{proposition}

The proof of uniqueness can be found in \cite{MO}, but details concerning the Lipschitz constant of  $h_{\beta A}$ are missing. Let us then present such details following ideas of  \cite{bousch} and \cite{BCLMS}.

\begin{proof}
 First, note that we can drop $\beta$ in the computations. For each  $0<s<1$, we define the operator $T_s: C(X,\mathbb{R})\to C(X,\mathbb{R}) $ by
$$T_s(u)(x) = \log \int_J e^{A(j,x)+su(\phi_j(x))}\, d\nu(j).$$

We have $|T_s(u)-T_s(v)|_{\infty} \leq s|u-v|_{\infty},$
meaning that, $T$ is a uniform contraction map on $C(X,\mathbb{R})$ with respect to the supremum norm ($|g|_{\infty}=\sup_{x\in X} |g(x)|$).
Let $u_s\in C(X,\mathbb{R})$ be the unique fixed point for $T_s$.

We claim that for any $s\in (0,1)$, the function $u_s$ is Lipschitz continuous, with $\lip(u_s) \leq \frac{\lip(A)}{1-\gamma}$. Indeed, for any $x_1,x_2 \in X$ we have
\[u_s(x_1) - u_s(x_2) \leq \max_j\{ A(j,x_1)-A(j,x_2) + s[u_s(\phi_j( x_1)) - u_s(\phi_j (x_2))]\}.\]
By iterating this inequality we get
\[u_s(x_1) - u_s(x_2) \leq \max_{(j_0,j_1,...)}\left(
\sum_{n=0}^{\infty}s^{n}[A(j_n,\phi_{j_{n-1}}\circ...\circ\phi_{j_0}(x_1)) -
A(j_n,\phi_{j_{n-1}}\circ...\circ\phi_{j_0}(x_2))]\right)\]
\[\leq \sum_{n=0}^{\infty}s^{n}Lip(A)\gamma^{n}{d(x_1,x_2)} = \frac{\lip(A)}{1-s\gamma}{d(x_1,x_2)} \leq \frac{\lip(A)}{1-\gamma}{d(x_1,x_2)}.\]

Particularly  $-\frac{\operatorname{Lip}(A)}{1-\gamma}\operatorname{diam}(X) \leq u_s-\max u_s \leq 0$ and then the family $\{ u^*_s=u_s-\max u_s \}_{0<s<1}$ is equicontinuous and uniformly bounded. Note also that, as $T_s(u_s)=u_s$, we obtain $$-|A|+ s \min u_s \leq u_s(x)\leq |A|+ s \max u_s, \,\forall x\in X.$$
Consequently,
$-|A| \leq (1-s) \min u_s\leq (1-s) \max u_s\leq |A|$, for any $0<s<1$.
Then, there exists a subsequence $s_n \to 1$ and a number $k\in \mathbb{R}$ such that $ [\,(1-s_n) \, \max u_{s_n}\,]\,\to k$. By applying Arzela-Ascoli theorem we get a subsequence, also denoted by $s_n$, such that  $\{u_{s_n}^*\}_{n\geq 1}$  converges uniformly to a function $u \in C(X,\mathbb{R})$. As $\lip(u_{s}^{*})=\lip(u_s) \leq \frac{\lip(A)}{1-\gamma}$ we get $\lip(u) \leq  \frac{\lip(A)}{1-\gamma}$.
In order to finish the proof, we will show that $h=e^{u}$ is an eigenfunction for $\mathcal{L}_{\phi,e^{A},\nu}$ with associated eigenvalue $e^{k}$. Indeed, for any $s\in(0,1)$ we have
\begin{eqnarray*}
	e^{u^*_s (x)}&=& e^{u_s(x) - \max u_s}=e^{-(1-s) \max u_s + u_s (x) - s \max u_s}\\
	&=&e^{-(1-s) \max u_s}\,\int_J e^{  A(j,x)+ s  u_s (\phi_j(x)) - s \max u_s}\, d\nu(j).
\end{eqnarray*}

Taking the limit when $s \to 1$ we conclude that $u$ satisfies
$$e^{u(x)} =e^{-k} \,  \int_J e^{  A(j,x)+   u (\phi_j(x))}\,d\nu(j)=e^{-k}\mathcal{L}_{\phi,e^{A},\nu}(e^u)(x).$$
\end{proof}

Consider the function $q^{\beta}:J\times X \to \mathbb{R}$ which is defined by $q^{\beta}(j,x):=\frac{e^{\beta A(j,x)}\cdot h_{\beta A}(\phi_j(x))}{ \lambda_{\beta A}\cdot h_{\beta A}(x)}$ and for each $j\in J$ consider also the function $q_j^{\beta}:X\to\mathbb{R},$ defined by $q_j^{\beta}(x) = q^{\beta}(j,x)$.

Replacing $e^{\beta A}$ by the function $q^{\beta}$ we get the operator $\mathcal{L}_{\phi,q^{\beta},\nu}$ which satisfies
\[\mathcal{L}_{\phi,q^{\beta},\nu}(1)=1,\]
that is, it has the eigenvalue $\lambda=1$ associated to the eigenfuntion $h=1$. We say that this operator $\mathcal{L}_{\phi,q^{\beta},\nu}$ is normalized.

The dual of the normalized operator $\mathcal{L}_{\phi,q^{\beta},\nu}$, which we will denote by $\mathcal{M}_{\phi,q^{\beta},\nu}$, acting in Borel probabilities on $X$ is given by
\[\mathcal{M}_{\phi,q^{\beta},\nu}(\mu)(f) = \int_X \mathcal{L}_{\phi,q^{\beta},\nu}(f)(x)\,d\mu(x) = \int_X \int_J {q^{\beta}(j,x)}\cdot f(\phi_j(x))\,d\nu(j)\,d\mu(x),\]
for any $f\in C(X,\mathbb{R}).$
There exists a unique probability $\rho_{_{\beta A}}$ (Gibbs probability) which is invariant for $\mathcal{M}_{\phi,q^{\beta},\nu}$, that is, it satisfies
\eqref{eq: rho invariant}.

We say that a probability $\pi$ on $J\times X$ is holonomic if
$$\int g(x)\,d\pi(j,x)=\int g(\phi_j(x))\,d\pi(j,x),\,\forall g\in C(X,\mathbb{R}).$$
Let us denote by $\mathcal{H}$ the set of holonomic probabilities on $J\times X$.

We have a natural holonomic probability $\pi_{\beta A}\in \mathcal{H}$ associated to $\beta A,$   defined by
\[\int g(j,x)\,d\pi_{\beta A}(j,x) :=  \int_X \int_J {q^{\beta}(j,x)}\cdot g(j,x)\,d\nu(j)\,d\rho_{_{\beta A}}(x),\,\,\forall g\in C(J\times X,\mathbb{R}).\]
This probability $\pi_{\beta A}$ is an equilibrium probability to $\beta A$ in the following sense:
\begin{equation}\label{pressure} P_\nu(\beta A):=\sup_{\pi \in \mathcal{H} } \int \beta A\, d\pi + H_{\nu}(\pi) = \int \beta A\, d\pi_{\beta A} + H_{\nu}(\pi_{\beta A}),
\end{equation}
where $H_{\nu}(\pi)$ is the relative entropy of $\pi$ with respect to $\nu$. More precisely, if the $x-$marginal of $\pi$ is a probability $\rho$ on $X$ then
\[H_{\nu}(\pi) =\left\{\begin{array}{cc}
	-\int \log(\operatorname{Jac})\,d\pi & \text{if}\, d\pi = \operatorname{Jac}(j,x)d\nu(j)d\rho(x)\\ \\
	-\infty & \begin{array}{c}\text{if}\,\pi\,\text{is not absolutely continuous}\\\text{with respect to}\,\nu\times\rho\end{array}\end{array}\right. . \]
The above defined entropy satisfies $H_{\nu}(\pi)=-D_{KL}(\pi\,|\, \nu\times \rho)$, where $D_{KL}$ is the Kullback-Leibler divergence (see \cite{LopM} for additional details).

We say that a measurable function $\operatorname{Jac}: J \times X \to [0,+\infty)$ is
a $\nu$-Jacobian, if $\int_J \operatorname{Jac}(j,x)\,d\nu(j) = 1,\,\forall x\in X$. The functions $1$ and $q^{\beta}$ are $\nu$-Jacobians. In \cite{LopM} it is proved that
\[H_{\nu}(\pi) = -\sup\left\{\int \log(\operatorname{Jac}) \,d\pi\,|\, \operatorname{Jac}\,\,\text{is a Lipschitz}\,\nu\text{-Jacobian}\right\}.\]
Consequently
\[H_{\nu}(\pi)\leq \int \log(1) \,d\pi=0.\]
Furthermore,
\[H_\nu(\pi_{\beta A}) = -\int \log(q^{\beta}) \,d\pi_{\beta A}.\]

This definition of entropy does not depend of the IFS, while the operator $\mathcal{L}$ depends of $\phi,$ as can be observed from its definition. An interesting fact is that  $P_\nu(\beta A) = \log(\lambda_{\beta A})$ (see \cite{MO}).

\subsection{Zero temperature limits and ergodic optimization }
 This section follows ideas present in  \cite{LMMS} and \cite{BLL} which were adapted to the present setting. It extends to IFS some results concerning zero temperature limits. The zero temperature case corresponds to an analysis of the limits, as $\beta\to +\infty,$ of the concepts introduced in above subsection (for instance: $\rho_{_{\beta A}},\,h_{\beta A}, \,\pi_{\beta A}, P_{\nu}(\beta A)$).

The following result generalizes one from \cite{LMMS} and will be useful in the present work.
\begin{lemma}\label{lmms}
	Let $Y,Z$ be compact metric spaces. Let $W_{\beta}:Y\times Z\to \mathbb{R}$ be a family of measurable functions converging uniformly to a continuous function $W:Y\times Z\to \mathbb{R}$, as $\beta\to +\infty,$ and let $\mu$ be a finite measure on $Y$ with $\supp(\mu) =Y$. Then
	\[\frac{1}{\beta} \log \int_Y e^{\beta W_\beta(y,z)} d\mu(y) \to \sup_{y\in Y}\, W(y,z)\]  uniformly on $Z,$ as $\beta\to+\infty$.
	The same is true if we replace $\beta$ by a sequence $\beta_i$ which converges to $+\infty$.
\end{lemma}

\begin{proof} As $W$ is continuous and $Y$ is compact we have that $\tilde{W}(z):=\sup_{y\in Y}\, W(y,z)$ defines a continuous function too.
	Fix an $\epsilon>0$. As $W,\tilde{W}$ are also uniformly continuous, there exists $\delta>0$ such that
	$$d(y_1,y_2)+d(z_1,z_2)<\delta \Rightarrow |W(y_1,z_1)-W(y_2,z_2)|<\epsilon$$ \text{and} $$d(z_1,z_2)<\delta \Rightarrow |\tilde{W}(z_1)-\tilde{W}(z_2)|<\epsilon.$$ As $W_\beta \to W$ uniformly we can suppose also there exists $\beta_0>0$ such that $$d(y_1,y_2)+d(z_1,z_2)<\delta \Rightarrow |W_\beta(y_1,z_1)-W_\beta(y_2,z_2)|<2\epsilon  ,\,\text{for}\, \beta>\beta_0.$$
	
	 Let $Z_0\subseteq Z$ be a finite set such that $\cup_{z\in Z_0} B(z,\delta) = Z$. For each $z\in Z$, let $y_z \in Y$ be such that $W(y_z,z) > \tilde{W}(z)-\epsilon$. Then, for any $z\in Z$ and $y \in B(y_z,\delta)$, we have $W(y,z) > \tilde{W}(z)- 2\epsilon$. As $W_\beta \to W$ uniformly, there exists $\beta_1>\beta_0$ such that
	$W_\beta(y,z) > \tilde{W}(z)- 3\epsilon$, for any $z\in Z$, $y \in B(y_z,\delta)$ and $\beta > \beta_1$.
	
	As $Z_0$ is a finite set and $\supp(\mu)=Y,$ there exists $\beta_2>\beta_1$ such that $$\frac{1}{\beta}\log\big(\mu\big(B(y_{z_0},\delta)\big)\big) >-\epsilon,\,\forall z_0\in Z_0,\, \beta>\beta_2.$$
	
	Given any $z\in Z$, there exists $z_0\in Z_0$ satisfying $d(z,z_0)<\delta$. If $\beta > \beta_2$ we have that	
	\[\frac{1}{\beta}\log\int_Y e^{\beta W_\beta(y,z)} d\mu(y) > \frac{1}{\beta}\log\int_Y e^{\beta (W_\beta(y,z_0)-2\epsilon)} d\mu(y) \geq\frac{1}{\beta}\log\int_{B(y_{z_0},\delta)} e^{\beta (W_\beta(y,z_0)-2\epsilon)} d\mu(y)\]
	\[ > \frac{1}{\beta}\log\left(\mu\big(B(y_{z_0},\delta)\big) e^{\beta(\tilde{W}(z_0)-5\epsilon)}\right)
 >\tilde{W}(z_0)-6\epsilon>\tilde{W}(z)-7\epsilon.\]
	
	On the other hand,  there exists $\beta_4>\beta_3$ such that, $W_\beta(y,z) < \tilde{W}(z)+ \epsilon$, for any $\beta>\beta_4$ and $y\in Y$. Then, for $\beta>\beta_4$ we have
	\[  \frac{1}{\beta} \log \int_Y e^{\beta W_\beta(y,z)} d\mu(y) <\tilde{W}(z)+\epsilon.\]
	
\end{proof}

\bigskip

 Remember that $\log(\lambda_{\beta A})=P_\nu(\beta A)$.

\bigskip

\begin{definition}
Let us define the number $\displaystyle{m(A):=\sup_{\pi \in \mathcal{H}}\int A\,d\pi}$.

	A continuous function $V:X\to\mathbb{R}$ is a calibrated subaction of $A$ if
	\[\sup_{j\in J} [A(j,x) +V(\phi_j(x)) - V(x) - m(A) ]=0,\,\forall x\in X.\]
	
	A holonomic probability $\pi \in \mathcal{H}$ is maximizing with respect to $A$ if $\int A\,d\pi = m(A)$.
\end{definition}	

\begin{proposition}
Consider a Lipschitz continuous function $A:J\times X\to\mathbb{R}$. Under the above framework, we have:\newline
\begin{enumerate}
\item $\displaystyle{\lim_{\beta\to +\infty}\frac{1}{\beta}\log(\lambda_{\beta A}) = \lim_{\beta\to +\infty}\frac{P_{\nu}(\beta A)}{\beta} = m(A).}$ \newline
\item Suppose that $\frac{1}{\beta_n}\log(h_{\beta_n A}) \to V$ uniformly, as $\beta_n\to+\infty$. Then $V$ is a Lipschitz calibrated subaction of $A$. \newline
\item Suppose that $\pi_{\beta_n A} \to \pi_{\infty}$ in the weak$^*$ topology, as $\beta_n\to+\infty$. Then the probability $\pi_{\infty}$ is holonomic and maximizing  to $A$.
\end{enumerate}
\end{proposition} 	
\begin{proof}
By definition of $h_{\beta A}$ and $\lambda_{\beta A}$ we have
 \begin{equation}\label{eq: lambda}
 \lambda_{\beta A} = \int_J e^{\beta A(j,x) +\log(h_{\beta A}(\phi_j(x)))-\log(h_{\beta A}(x))}\,d\nu(j),\,\,\forall x\in X.
 \end{equation}
If $h_{\beta A}$ attains its minimum and maximum in $x_1$ and $x_2$, respectively, we get:
\[\lambda_{\beta A} = \int_J e^{\beta A(j,x_2) +\log(h_{\beta A}(\phi_j(x_2)))-\log(h_{\beta A}(x_2))}\,d\nu(j)\leq \int_J e^{\beta A(j,x_2)}\,d\nu(j) \leq e^{\beta\cdot \max_{j,x}A(j,x)}\]
	and similarly
	\[	\lambda_{\beta A} \geq \int_J e^{\beta A(j,x_1)}\,d\nu(j) \geq e^{\beta\cdot \min_{j,x}A(j,x)}.\]
Then
\[			\min_{j,x}A(j,x)\leq \frac{1}{\beta}\log(\lambda_{\beta A}) \leq \max_{j,x}A(j,x).\]
Let $k$ be an accumulation point of $\frac{1}{\beta}\log(\lambda_{\beta A})$ as $\beta\to +\infty$ and suppose that $\frac{1}{\beta_n}\log(\lambda_{\beta_n A}) \to k$. By Proposition \ref{prop: eigen} the family $\frac{1}{\beta}\log(h_{\beta A})$ is equicontinuous and uniformly bounded. By applying Arzela-Ascoli theorem, we get a subsequence of $(\beta_n)$, also denoted by $(\beta_n),$ and a Lipschitz continuous function $V$ such that  $\frac{1}{\beta_n}\log(h_{\beta_n A})\to V$ uniformly.
By taking $\lim_{\beta_n \to +\infty}\frac{1}{\beta_n}\log$ in both sides of \eqref{eq: lambda} and applying Lemma \ref{lmms} we obtain
\[k = \sup_{j}[ A(j,x) +V(\phi_j(x)) - V(x)] ,\,\,\forall x\in X.\]

We claim that $k=m(A)$. Indeed, by one hand, as $\pi_{\beta A}$ is holonomic,  we have
\[	\frac{1}{\beta}\log(\lambda_{\beta A}) =	\frac{P_\nu(\beta A)}{\beta} = 	\frac{\int \beta A\,d\pi_{\beta A} + H_{\nu}(\pi_{\beta A})}{\beta}
 \stackrel{H_\nu \leq 0}{\leq} 	\frac{\int \beta A\,d\pi_{\beta A} }{\beta}= 	\int A\,d\pi_{\beta A} \leq m(A).\]
Therefore $k\leq m(A)$. On the other hand, if $\pi$ is any holonomic probability, as $k\geq A(j,x) +V(\phi_j(x)) - V(x)$ we get $k \geq  \int A\,d\pi$. Therefore, $k\geq m(A)$.

We conclude that $k=m(A)$ for any possible convergent sequence and then $$\lim_{\beta\to+\infty}\frac{1}{\beta}\log(\lambda_{\beta A}) = m(A).$$ Clearly, we also get
\[m(A) = \sup_{j}[ A(j,x) +V(\phi_j(x)) - V(x) ],\,\,\forall x\in X,\]
(and then $V$ is a calibrated subaction) for any uniform limit of $\frac{1}{\beta}\log(h_{\beta A})$, as $\beta\to\infty$.

Finally, suppose that $\pi_{\beta_n A} \to \pi_{\infty}$ in the weak$^*$ topology, as $\beta_n\to+\infty$. Then $\pi_{\infty}$ is holonomic and $\int A\,d\pi_{\infty} \leq m(A)$. On the other hand,
\[m(A)  =\lim_{\beta_n\to+\infty}	\frac{P_\nu(\beta_n A)}{\beta}
\stackrel{H_\nu \leq 0}{\leq} \lim_{\beta_n\to+\infty}	\int  A\,d\pi_{\beta_n A} = 	\int A\,d\pi_{\infty}.\]

\end{proof}

Given $x \in X$, $n\in\mathbb{N}$ and a finite sequence $\omega = (j_1,j_2,...,j_n)\in J^{n}$ we denote
$$\phi_{\omega}(x)=\phi_{(j_1,...,j_n)}(x) := \phi_{j_1}\circ\dots\circ \phi_{j_n}(x)$$
and
$${\rm Sum}(A,\omega,x) :=  A({j_1},\phi_{(j_2,...j_n)} (x))+ A(j_2,\phi_{(j_3,...,j_n)} (x))+\dots+ A(j_n,x)-n\cdot m(A).$$

Given $x,y \in X$ and $\varepsilon>0$ we define
\begin{equation}\label{def: S_epsilon A} S_{A,\varepsilon}(x,y) = \sup_{n\in\mathbb{N}}\left[\sup_{\omega \in J^n \,|\,d(x,\phi_{\omega}(y))<\varepsilon}{\rm Sum}(A,\omega,y)\right],
\end{equation}
which can be $-\infty$ if the set $\{\omega \in J^n\,|\,d(x,\phi_{\omega}(y))<\varepsilon\}$ is empty for any $n$. Clearly $0<\epsilon < \epsilon' \Rightarrow S_{q,\varepsilon}\leq S_{q,\varepsilon '}$, so we can define the \textit{Ma\~{n}\'{e} potential} $S_A:X\times X \to \mathbb{R}\cup\{-\infty\},$ by
\begin{equation}\label{eq: Mane}
	S_A(x,y) =\lim_{\varepsilon\to 0} S_{A,\varepsilon}(x,y).
\end{equation}
The  \textit{Aubry set} is defined as
\begin{equation}\label{eq: Aubry}
	\Omega_A = \{x\in X | S_A(x,x) = 0\}.
\end{equation}

 Let $V$ be a calibrated subaction to $A$ and define $q(j,x):= A(j,x)+V(\phi_j(x))-V(x)-m(A)$. Then, by definition of calibrated subaction, we have $\sup_{j\in J} q(j,x) = 0,\,\forall\, x\in X$. For any holonomic probability $\pi$ on $J\times X$ we have
$\int q\,d\pi = \int A\,d\pi -m(A)$. Then $m(q) = 0$ and a holonomic probability $\pi$ is maximizing to $A$ if and only if it is maximizing to $q$.

\begin{theorem}\label{Mane ergodic}
	Let $A:J\times X\to\mathbb{R}$ be a Lipschitz function and $V$ be any continuous calibrated subaction to $A$. Let us define the function $q:J\times X \to \mathbb{R}$ by $q(j,x):= A(j,x)+V(\phi_j(x))-V(x)-m(A)$. Then $S_q(x,y) = S_A(x,y)+V(x)-V(y)$ and $\Omega_q=\Omega_A,$ which is a non empty set. Furthermore,
		\begin{equation}\label{eq : V Mane}
			V(x) = \sup_{z\in \Omega_A} [S_A(z,x) +V(z)],\,\,\forall x\in X.
		\end{equation}
\end{theorem}
\begin{proof}
We have $m(q)=0$, then, for $\omega=(j_1,...,j_n)$,
	$${\rm Sum}(q,\omega,y) :=  q(j_1,\phi_{(j_2,...j_n)} (y))+ q({j_2},\phi_{(j_3,...,j_n)} (y))+\dots+ q(j_n,y).$$	
	As $q(j,x)= A(j,x)+V(\phi_j(x))-V(x)-m(A)$ we have
	$${\rm Sum}(q,\omega,y) = {\rm Sum}(A,\omega,y) + V(\phi_{\omega}(y))-V(y).$$
	Then
	\[S_q(x,y) =\lim_{\varepsilon\to 0}\sup_{n\in\mathbb{N}}\left[\sup_{\omega \in J^n \,|\,d(x,\phi_{\omega}(y))<\varepsilon}({\rm Sum}(A,\omega,y) + V(\phi_{\omega}(y))-V(y))\right]\]
	\[ =S_A(x,y)+V(x)-V(y).\]
	
The Aubry sets satisfy $\Omega_A=\{x\in X\,|\,S_A(x,x)=0\}$ and $\Omega_q=\{x\in X\,|\,S_q(x,x)=0\}$. As $S_A(x,x)=S_q(x,x)$ we obtain $\Omega_A = \Omega_q.$

Let us denote also $q_j(x)=q(j,x)$.	We still need to prove that $\Omega_A\neq\emptyset$ and
$\sup_{z\in \Omega_A} S_q(z,x) =0,\,\forall x\in X.$
As $V$ is a calibrated subaction we have $\sup_{j\in J} q(j,x) = 0,\forall x\in X$ and so $S_q\leq 0$.
From now on we fix $x\in X$ and prove that there exists $\tilde{x}\in \Omega_A$ such that $S_q(\tilde{x},x)=0$.

Let $(j_n)$ be a sequence of points of $J$ satisfying $q_{j_1}(x) = 0 $ and $q_{j_{n+1}}(\phi_{j_n}\circ...\circ\phi_{j_1}(x)) = 0$ for all $n\geq 1$.  Let $x_n :=\phi_{(j_n,...,j_1)}(x)= \phi_{j_n}\circ...\circ\phi_{j_1}(x)$.

As $X$ is compact there exists a subsequence $n_i$ and a point $\tilde{x}$ such that $x_{n_i}\to \tilde{x}$. It follows from the definition of $S_q(\cdot,\cdot)$ that  $S_q(\tilde{x},x)\geq \lim_{n_i\to+\infty} {\rm Sum}(q,(j_{n_i},...,j_1),x)= 0$ and then $S_q(\tilde{x},x)=0$. In what follows we will prove that $\tilde{x} \in \Omega$.

\bigskip

 Let $\omega_{k,l}:=(j_{k},j_{k-1},...,j_{l})$ for $k\geq l$. By construction of $(j_n)$ we have $q_{j_{n+1}}(x_n) = 0$ and
\begin{equation}\label{eq1_aubry}
	{\rm Sum}(q,\omega_{k,n+1},x_n)=0,\,\,\,\forall k>n.
\end{equation}

Given $\varepsilon>0$ there exist $n,m\in\mathbb{N}$ such that $n<m$ and $x_n, x_m \in B(\tilde{x},\varepsilon/2)$. Observe that $x_m = \phi_{\omega_{m,n+1}}(x_n)$ and as $\phi$ satisfies \eqref{eq: gamma contraction}   we get 	
\[d(\tilde{x},\phi_{\omega_{m,n+1}}(\tilde{x})) \leq d(\tilde{x},x_m)+ d(x_m, \phi_{\omega_{m,n+1}}(\tilde{x}))\]\[ =d(\tilde{x},x_m)+ d(\phi_{\omega_{m,n+1}}(x_n), \phi_{\omega_{m,n+1}}(\tilde{x}))<\varepsilon.\]

It follows that
\[S_{q,\varepsilon}(\tilde{x},\tilde{x})\geq {\rm Sum}(q,\omega_{m,n+1},\tilde{x}).\]

As $A$ is Lipschitz continuous and $\phi$ satisfies \eqref{eq: gamma contraction}, there exists $C>0$ such that
\[{\rm Sum}(q,\omega_{m,n+1},\tilde{x}) =  {\rm Sum}(A,\omega_{m,n+1},\tilde{x}) + V(\phi_{\omega_{m,n+1}}(\tilde{x}))-V(\tilde{x})\]
\[\geq {\rm Sum}(A,\omega_{m,n+1},x_n)- \frac{C}{1-\gamma}\frac{\varepsilon}{2}  + V(\phi_{\omega_{m,n+1}}(\tilde{x}))-V(\tilde{x})  \]
\[={\rm Sum}(q,\omega_{m,n+1},x_n)- \frac{C}{1-\gamma}\frac{\varepsilon}{2}+[V(\phi_{\omega_{m,n+1}}(\tilde{x}))-V(\phi_{\omega_{m,n+1}}(x_n))]-[V(\tilde{x}) -V(x_n)].\]
Therefore, applying also equation \eqref{eq1_aubry} and using that $V$ is continuous, we have
 $S_q(\tilde{x},\tilde{x}) =\displaystyle\lim_{\varepsilon\to 0} S_{q,\varepsilon}(\tilde{x},\tilde{x}) \geq 0.$ As we always have  $S_q\leq 0$,  we complete the proof.

\end{proof}

\section{On invariant idempotent probabilities}\label{sec: idempotent}
In this section we recall definitions and results containing in \cite{MO2}. The theory of idempotent probabilities was introduced by Maslov (also called Maslov measures)  to model problems of optimization as integrals with respect to a max-plus structure, see \cite{KM89}, \cite{Kol97}, \cite{MoDo98} and \cite{LitMas96} for classical results on Maslov measures, also  \cite{MZ}, \cite{ZAI}, \cite{Zar}, \cite{dCOS20fuzzy} and \cite{ExistInvIdempotent} for the existence of a unique invariant idempotent probability for non-place dependent IFSs.

Let us consider $\mathbb{R}_{\max}:=\mathbb{R}\cup \{-\infty\}$ endowed with the operations
\begin{enumerate}
	\item $\oplus: \mathbb{R}_{\max} \times \mathbb{R}_{\max} \to \mathbb{R}_{\max}$, where $a \oplus b :=\max(a,b)$ assuming $a \oplus -\infty:=a$.
	\item $\odot: \mathbb{R}_{\max} \times \mathbb{R}_{\max} \to \mathbb{R}_{\max}$, where $a \odot b :=a+b$ assuming $a \odot -\infty:=-\infty$.
\end{enumerate}
We consider the following conventions and notations: for any function $f:A\to\mathbb{R}_{\max}$,
\begin{enumerate}
	\item $\bigoplus_{a \in A } := \sup_{a \in A}$;
	\item $\bigoplus_{a \in A}f(a) = -\infty$, if $A=\emptyset$;
	\item $\bigoplus_{a \in A}f(a) = -\infty$, if $f(a)=-\infty,\,\forall a\in A$ .
	\item $\bigoplus_{a \in A}f(a) = +\infty$, if $A\neq \emptyset$ and $\{f(a)\,|\,a\in A\}$ is not bounded above.
\end{enumerate}

\begin{definition}\label{def: Idempotent probability as a functional}
	A functional $m: C(X,\mathbb{R}) \to \mathbb{R}$ is an idempotent probability on $X$ if
	\begin{enumerate}
		\item  $m(a \odot f)=a \odot m( f)$, $\forall a \in \mathbb{R}$ and $\forall f \in C(X,\mathbb{R})$;
		\item  $m(f \oplus g)=m( f) \oplus m( g)$, $\forall f, g \in C(X,\mathbb{R})$;
		\item  $m(0)=0$.
	\end{enumerate} We denote by $I(X)$ the set of idempotent probabilities on $X$.
\end{definition}

The idempotent probabilities are closely related with upper semi-continuous  functions (u.s.c.). We denote by $U(X, \mathbb{R}_{\max})$  the set of u.s.c. functions $\lambda:X\to \mathbb{R}_{\max}$ satisfying $\{x \in X | \lambda(x) \neq -\infty\} \neq \varnothing.$
The next result is well known in the literature (mainly from \cite{KM89}) and a proof for the present setting can be founded in \cite{MO2}.

\begin{theorem}\label{teo : densidade Maslov} A functional $\mu:C(X, \mathbb{R}) \to  \mathbb{R} $ is an idempotent probability if and only if there exists $\lambda\in U(X, \mathbb{R}_{\max})$ satisfying
	\[\mu(\psi) = \bigoplus_{x\in X} \lambda(x)\odot \psi(x),\,\,\,\,\forall \,\psi \in C(X,\mathbb{R})\,\,\,\text{and}\,\,\,\bigoplus_{x\in X}\lambda(x) = 0. \]
	If $\mu\in I(X)$ there is a unique such function $\lambda$ in $U(X, \mathbb{R}_{\max})$, which is called its density.
\end{theorem}

\begin{definition} \label{def:mpIFS with measures}
	Let $(X,d_X)$ and $(J,d_J)$ be compact metric spaces.\,	A max-plus IFS (mpIFS for short) $\mathcal{S}=(X, (\phi_j)_{j\in J}, (q_j)_{j\in J})$ is a uniformly contractive IFS $(X, (\phi_j)_{j \in J})$ endowed with a normalized family of weights $(q_j)_{j\in J}$, which is a family of functions $\{q_j:X\to\mathbb{R}\,|\,j\in J\}$ satisfying: \newline
	i. \, there exists a constant $C>0$ such that
	\begin{equation}\label{eq: q lipschitz}
		|q_{j}(x_1)-q_{j}(x_2)|\leq C \cdot d_X(x_1,x_2),\,\,\forall j\in J,\; \forall x_1,\,x_2\in X;
	\end{equation}
	ii. \,
	\begin{equation}\label{eq: q normalized}
		\bigoplus_{j\in J} q_j(x) = 0, \,\,\forall x\in X;
	\end{equation}
	iii. \, the function $q:J\times X\to\mathbb{R}$, defined by $q(j,x)=q_j(x)$, is continuous.
\end{definition}

We also use the compact notation $\mathcal{S}=(X, \phi, q)$ to denote a mpIFS (assuming that the set of indices $J$ is known). Considering the setting of above section, if $A:J\times X\to\mathbb{R}$ is Lipschitz continuous and $V:X\to\mathbb{R}$ is a Lipschitz calibrated subaction of $A$, then defining $q(j,x) =A(j,x)+V(\phi_j(x))-V(x)-m(A)$ we get a normalized family of weights. This is the case of $\displaystyle q(j,x) =\lim_{\beta \to +\infty} \frac{1}{\beta}\log(q^{\beta}(j,x))$, where $q^{\beta}(j,x) = \frac{e^{\beta A(j,x)} \cdot h_{\beta A}(\phi_j(x))}{h_{\beta A}(x)\cdot \lambda_{\beta A}}$, when such limit exists.

\begin{definition}\label{def:mp ruelle operator}
	To each mpIFS $\mathcal{S}=(X, \phi, q)$ we assign the following operators:\newline
	1.  $\mathcal{L}_{\phi,q}^{0}: C(X, \mathbb{R})   \to C(X, \mathbb{R})  $, defined by
	\begin{equation}\label{eq:mp ruelle operator}
		\mathcal{L}_{\phi,q}^0(f)(x):=\bigoplus_{j \in J} q_{j}(x)\odot f(\phi_{j}(x)).
	\end{equation}
	2.  $\mathcal{M}_{\phi,q}^0: I(X) \to I(X),$ defined by
	\begin{equation}\label{eq:Markov operator}
		\mathcal{M}_{\phi,q}^0(\mu):= \bigoplus_{j \in J}\mu( q_{j}\odot (f\circ \phi_{j})).
	\end{equation}
	3.  $L_{\phi,q}^0: U(X, \mathbb{R}_{\max}) \to U(X, \mathbb{R}_{\max}),$ defined by
	\begin{equation}\label{eq:transfer operator}
		L_{\phi,q}^0(\lambda)(x):=\bigoplus_{(j,y)\in \phi^{-1}(x)} q_{j}(y) \odot  \lambda(y).
	\end{equation}
\end{definition}

The operators $\mathcal{L}_{\phi,q}^0$ and $\mathcal{M}_{\phi,q}^0$ play the hole of the operators $\mathcal{L}_{\phi,q^{\beta},\nu}$  and $\mathcal{M}_{\phi,q^{\beta},\nu}$, which were introduced in section \ref{sec:Thermodynamic}, when considered the zero temperature case and the scale $\frac{1}{\beta}\log$ (see also Lemma \ref{lmms}).

Next theorem establishes the relation between the three operators in Definition \ref{def:mp ruelle operator}.

\begin{theorem}\label{thm: equivalences} Given a function $\lambda \in  U(X, \mathbb{R}_{\max})$ satisfying $\oplus_x \lambda(x) = 0$ and the associated idempotent probability   $\mu=\bigoplus_{x\in X}\lambda(x)\odot\delta_x\in I(X)$ we have that $\mathcal{M}_{\phi,q}^0(\mu)= \bigoplus_{x \in X} L_{\phi,q} (\lambda)(x) \odot \delta_x$, that is, $\mathcal{M}_{\phi,q}^0(\mu)$ has density $L_{\phi,q} (\lambda)$ where $\lambda$ is the density of $\mu$. Furthermore
	$$\mathcal{M}_{\phi,q}^0(\mu)(f) = \mu(\mathcal{L}_{\phi,q}^0(f)),$$
	for any $f \in C(X, \mathbb{R})$, that is, $\mathcal{M}_{\phi,q}^0$ is the max-plus dual of $\mathcal{L}_{\phi,q}^0$.
\end{theorem}

\begin{definition}\label{def:Markov operator}
	An idempotent probability  $\mu \in I(X)$ with density $\lambda \in  U(X, \mathbb{R}_{\max})$ is called invariant (with respect to the mpIFS) if it satisfies any of the following equivalent conditions:\newline
	1. $\mathcal{M}_{\phi,q}^0(\mu)=\mu$;\newline
	2. $L_{\phi,q}(\lambda)=\lambda$;\newline
	3. $\mu(\mathcal{L}_{\phi,q}^0(f))=\mu(f)$,  for any $f \in C(X, \mathbb{R})  $.
\end{definition}

In \cite{MO2} the idempotent invariant probabilities were characterized using  the Ma\~{n}\'{e}  potential and the Aubry set.  Let us present these results below.

\begin{theorem} \label{teo: main result}
	
	Consider any function $\overline{\lambda}:X\to[-\infty,0]$ satisfying $\bigoplus_{x\in X}\bar{\lambda}(x) = 0$ and
	\[	 \bar{\lambda}(x) = \bigoplus_{z\in \Omega_q} [S_q(x,z) \odot \bar{\lambda}(z)].\]
	Then, $\overline{\mu}(f) :=\bigoplus_{x\in X} (\overline{\lambda}(x)\odot f(x))$ is an idempotent probability  which is invariant for the operator $\mathcal{M}_{\phi,q}^0$. Reciprocally, if $\mu \in I(X)$ is invariant for $\mathcal{M}_{\phi,q}^0$ and $\lambda \in U(X,\mathbb{R}_{\max})$ is its (unique u.s.c.) density, then $\lambda$ satisfies the equation.
	\begin{equation}\label{eq: represent density}
		\lambda(x) = \bigoplus_{z\in \Omega_q} [S_q(x,z) \odot \lambda(z)].
	\end{equation}
\end{theorem}

\bigskip
We remark that we have a symmetric correspondence between the equation which characterizes the calibrated subactions in Theorem \ref{Mane ergodic} and the equation which characterizes densities of invariant idempotent probabilities in Theorem \ref{teo: main result}. The main differences are the use of $S_A$ instead $S_q$ and the order of variables.

\bigskip
In \cite{MO2} it is presented an example proving that a max-plus IFS can have infinitely many possible idempotent invariant probabilities. On the other hand, if $\phi_j(x)=\phi_j$ does not depend of $x,$ the idempotent invariant probability is unique. This will play an important hole in Section \ref{sec:Large Deviation} and we present this result below.

Given any $\omega=(j_1,j_2,...) \in J^{\mathbb{N}}$, there exists a unique point, denoted by $\pi(\omega) \in X$ satisfying
\[ \pi(\omega) = \lim_{n\to+\infty} \phi_{j_1}\circ...\circ\phi_{j_n}(x) \]
for any point $x$  of $X$.

\begin{theorem}\label{thm:represent idemp invar cte}  Given a mpIFS, $S=(X,\phi,q)$, if  $q_j(x)=q_j$ does not depend on $x$, then there exists a unique invariant idempotent  probability for $\mathcal{M}_{\phi,q}^0$. Let $\lambda \in U(X,\mathbb{R}_{\max})$ be its density. Then for any $z \in \Omega$ we have
	\[\lambda(x) = S_q(x,z) = \bigoplus_{(j_1,j_2,j_3,...)\in\pi^{-1}(x)}[q_{j_1}+q_{j_2}+q_{j_3}+...].\]
\end{theorem}

\section{Zero temperature limits and Large Deviations}\label{sec:Large Deviation}

In \cite{Aki99}  it is presented  a necessary and sufficient condition for a family of probabilities to satisfy a large deviation principle through the notion of idempotent measures introduced by Maslov. In the present section we will prove that the idempotent probabilities invariant for $\mathcal{M}_{\phi,q}$ are related with large deviation principles for the measures $\rho_{_{\beta A}}$ at zero temperature. This will provide a connection between Sections \ref{sec:Thermodynamic} and \ref{sec: idempotent}.

\begin{definition}
	Let $(\mu_\beta)_{\beta>0}$ be a family of probabilities on $X$. We will say that $(\mu_\beta)$ satisfy a large deviation principle (LDP) if there exists a lower semi-continuous (l.s.c.) rate function $I: X \rightarrow[0,+\infty]$ such that\\
	1. for each closed set $C \subset X$
	$$
	\limsup_{\beta \to + \infty } \frac{1}{\beta} \log \mu_\beta(C) \leq-\inf _{\omega \in C} I(\omega);
	$$
	2. for each open set $U \subset X$
	$$
	\liminf_{\beta \to + \infty } \frac{1}{\beta} \log \mu_\beta(U) \geq-\inf _{\omega \in U} I(\omega).
	$$
\end{definition}

 Remark: As $X$ is open and closed and $\mu_\beta(X)=1$, we obtain $\inf_{w\in X}I(w) = 0.$
	Furthermore, we assume $\log(0)=-\infty,$  $\displaystyle{\lim_{\beta\to+\infty} \frac{1}{\beta}\log(0)=-\infty}$ and $\displaystyle{-\inf_{x\in \emptyset} I(x) = -\infty}$.

Next two results are very useful in the present work.

\begin{lemma}\label{lemma: Varadhan}[Varadhan's Lemma]
	Let $(\mu_\beta)_{\beta>0}$ be a family of probabilities on the compact metric space $X$ which satisfies a LDP with rate function $I: X \rightarrow[0,+\infty]$. Then, for any function $f\in C(X,\mathbb{R})$ we have
	\[\lim_{\beta\to\infty}\frac{1}{\beta}\log (\rho_{_{\beta A}}(e^{\beta {f}})) = \sup_{x\in X}[f(x)-I(x)].\]
\end{lemma} 	
\begin{proof} See  Theorem 4.3.1 in \cite{DZ}.\end{proof}

\begin{lemma}\label{lemma: Bric}[Bric's Inverse Varadhan's Lemma]
	Let $(\mu_\beta)_{\beta>0}$ be a family of probabilities on the compact metric space $X$. Suppose that for any function $f\in C(X,\mathbb{R})$ there exists the limit
	\[\Gamma(f):=\lim_{\beta\to\infty}\frac{1}{\beta}\log (\rho_{_{\beta A}}(e^{\beta {f}})).\]
	Then $(\mu_\beta)$ satisfies a LDP. Furthermore, denoting by $I: X \rightarrow[0,+\infty]$ its deviation function, we have
	\[\Gamma(f) = \sup_{x\in X}[f(x)-I(x)],\,\forall \,f\in C(X,\mathbb{R}).\]
\end{lemma}
\begin{proof} See Theorem 4.4.2 in  \cite{DZ}.\end{proof}

\subsection{Irreducible Aubry set and a large deviation principle}\label{sec: irreducible}

In this subsection we investigate the uniqueness of idempotent invariant probabilities through the Aubry set and proves a LDP.

\begin{definition} We will say that $\Omega_A$ is irreducible if $S_A(x,y)= 0$ for any $x,y\in \Omega_A$.
\end{definition}

As an example, in \cite{MO2} it is proved that $\Omega_q$ is irreducible if $q_j(x) = q_j$ does not depend of $x$.

\begin{corollary} If $\Omega_A$ is irreducible then any calibrated subaction is constant on $\Omega_A$ and any two calibrated subactions are equal, up to the addition of a constant.
\end{corollary}
\begin{proof}
	For any $x \in \Omega_A$ and any calibrated subaction $V$, applying Theorem \ref{Mane ergodic}, we have
	\[V(x) = \sup_{z\in \Omega_A} S_A(z,x) +V(z)=  \sup_{z\in \Omega_A} V(z).\]
	This proves that $V$ is constant in $\Omega_A$. Let us call $V_\Omega:=V(z),\,\forall z\in \Omega_A$. Then, for any $x\in X$ and any calibrated subaction $V$, we have
	\[ V(x) = \left(\sup_{z\in \Omega_A} S_A(z,x)\right) +V_\Omega.\]
	In this way if $V$ and $U$ are any two calibrated subactions, we have $V(x)-U(x) =V_\Omega -U_\Omega,\;\forall\,x\in X$.
\end{proof}

\begin{corollary} If  $\Omega_A$ is irreducible then there exists the uniform limit $$\displaystyle{V=\lim_{\beta \to +\infty} \frac{1}{\beta}\log(h_{\beta A})}$$ 	and consequently the uniform limit $\displaystyle{q=\lim_{\beta \to +\infty} \frac{1}{\beta}\log(q^{\beta})}$, which satisfies $q(j,x) = A(j,x)+V(\phi_j(x))-V(x)-m(A)$.
\end{corollary}
\begin{proof} By construction of $h_{\beta A}$ we have $\sup_{x\in X} h_{\beta A}(x) =1$ and the family $\frac{1}{\beta}\log(h_{\beta A})$ is equicontinuous  (see Proposition \ref{prop: eigen}). Let $V$ be the unique calibrated subaction satisfying $\sup_{x\in X} V(x) =0$ (any two of they differ just by a constant). By applying Arzela-Ascoli theorem we get  $\frac{1}{\beta}\log(h_{\beta A}) \to V$ uniformly.
\end{proof} 	

\begin{proposition}\label{prop: irreducible unique invariant} It suppose that $\Omega_A$ is irreducible. Let $q=\lim_{\beta \to +\infty} \frac{1}{\beta}\log(q^{\beta})= A(j,x)+V(\phi_j(x))-V(x)-m(A)$. There exists a unique invariant idempotent probability for the mpIFS $(X,\phi,q)$. Its density $\lambda$ satisfies $\lambda(x) = S_q(x,z),\,\,\forall z \in \Omega_A$.
\end{proposition}
\begin{proof}
	From Theorem \ref{Mane ergodic} we have $\Omega:=\Omega_A=\Omega_q$. From Theorem \ref{teo: main result}, if $\rho$ is an invariant idempotent probability with density $\lambda$  we have
	\[\lambda(x) = \sup_{z\in \Omega} S_q(x,z) +\lambda(z), \,\forall x\in X.\]
	If $x\in \Omega$ then we obtain $\lambda(x) = \sup_{z\in \Omega} \lambda(z)=:\lambda_\Omega$ is a constant.
	If $x\in X$ we have then
	\[\lambda(x) = \lambda_\Omega + \sup_{z\in \Omega} S_q(x,z).\]
	As $\sup_{x\in X}\lambda(x) = 0$ and $\sup_{x\in X}\sup_{z\in \Omega} S_q(x,z)= 0$ we get $ \lambda_\Omega=0$ and then
	\[\lambda(x) = \sup_{z\in \Omega} S_q(x,z).\]
	In \cite{MO2} it is proved that $S_q(x,z)\geq S_q(x,y)+S_q(y,z), \,\forall\,x,y,z\in X$. Then, if $z_1\in\Omega$,  we have $\lambda(x) = S_q(x,z_1)$, because
	\[\lambda(x) \geq  S_q(x,z_1) \geq \sup_{z\in \Omega} [S_q(x,z)+S_q(z,z_1) ]=  \sup_{z\in \Omega} S_q(x,z)=\lambda(x). \]
	This proves that the invariant idempotent probability for the mpIFS $(X,\phi,q)$ is unique and its density is $\lambda(x) = S_q(x,z_1),\,\forall z_1\in\Omega$.
\end{proof}

\begin{theorem}\label{teo : irreducible} Suppose that $\Omega_A$ is irreducible. Let $$q=\lim_{\beta \to +\infty} \frac{1}{\beta}\log(q^{\beta})= A(j,x)+V(\phi_j(x))-V(x)-m(A).$$ Then $(\rho_{_{\beta A}})$ satisfy a LDP with rate function $$I(x) = -S_q(x,z)=-S_A(x,z)-V(x)+V(z),\,\forall z \in \Omega_A.$$ Furthermore,
	\[\rho(f):= \sup_{x\in X} [f(x)-I(x)],\] defines an idempotent probability with density $-I$, which is the unique invariant to the mpIFS $(X,\phi,q).$ For any continuous function $f \in C(X,\mathbb{R})$ we have
	\[\lim_{\beta\to\infty}\frac{1}{\beta}\log (\rho_{_{\beta A}}(e^{\beta f})) = \rho(f).\]
\end{theorem}

Before we present the proof of Theorem~\ref{teo : irreducible}, we need two lemmas.

\begin{lemma}\label{lemma :tec1}  Let $g:J\times X\to\mathbb{R}$ be a continuous function and for each $\beta>0$, let $g_\beta: J\times X \to\mathbb{R}$ be a continuous function. Suppose that $g_\beta \to g$ uniformly and that  $\frac{1}{\beta}\log(\int_X e^{\beta\cdot g(j,x)}\,d\rho_{\beta A}(x))$ converges  to a function $\Gamma^*(j)$ uniformly on $J$. Then  $$\frac{1}{\beta}\log(\int_X e^{\beta\cdot g_\beta(j,x)}\,d\rho_{\beta A}(x))\to\Gamma^*(j) $$ uniformly on $J$.
	The same is true if we replace $\beta$ by an increasing sequence $\beta_i$ which converges to $+\infty$.
\end{lemma}
\begin{proof}  For any $\epsilon>0$ there exists $\beta_0$ such that
	$g -\epsilon <g_\beta<g +\epsilon$, $\beta>\beta_0$
	and
	\[ \Gamma^*(j)-\epsilon <  \frac{1}{\beta}\log\left(\int_X e^{\beta\cdot g(j,x)}\,d\rho_{\beta A}(x)\right) < \Gamma^*(j)+\epsilon,\,\,\forall \beta>\beta_0.\]
	Then we have, for $\beta>\beta_0$,
	\[\Gamma^*(j)-2\epsilon < \frac{1}{\beta}\log\left(\int_X e^{\beta\cdot g(j,x)}\,d\rho_{\beta A}(x)\right)-\epsilon = \frac{1}{\beta}\log\left(\int_X e^{\beta\cdot g(j,x) -\epsilon}\,d\rho_{\beta A}(x)\right)\]
	\[\leq \frac{1}{\beta}\log\left(\int_X e^{\beta\cdot g_\beta(j,x)}\,d\rho_{\beta A}(x)\right) <\frac{1}{\beta}\log\left(\int_X e^{\beta\cdot g(j,x) +\epsilon}\,d\rho_{\beta A}(x)\right)\]
	\[=\frac{1}{\beta}\log\left(\int_X e^{\beta\cdot g(j,x)}\,d\rho_{\beta A}(x)\right)+\epsilon<\Gamma^*(j)+2\epsilon.\]
\end{proof}

\begin{lemma} \label{lemma: tec2}
	Suppose for any $f:X\to\mathbb{R}$ there exists the limit $$\Gamma(f):=\lim_{\beta\to+\infty} \frac{1}{\beta}\log\left(\int_X e^{\beta f(x)}\,d\rho_{\beta A}(x)\right).$$ Let $q:J\times X\to\mathbb{R}$ be a continuous function and consider the function $\Gamma^*:J\to\mathbb{R}$, defined by $\Gamma^*(j)= \Gamma(q_j+f\circ \phi_j)$, where $q_j(x)=q(j,x)$.  Then  $\frac{1}{\beta}\log(\int_X(e^{\beta (q + f\circ \phi)})d\rho_{\beta A})$ converges uniformly to $\Gamma^*$ on $J$. 	The same is true if we replace $\beta$ by an increasing sequence $\beta_i$ which converges to $+\infty$.
\end{lemma}
\begin{proof}
	Fix any $\epsilon>0$. As $q$ and $f$ are continuous, $\phi$ satisfies \eqref{eq: gamma contraction} and $J$ is compact,  there is $\delta>0$ such that $d(j,l)<\delta \Rightarrow |q_j(x)-q_l(x) +f(\phi_j(x))-f(\phi_l(x))|<\epsilon,\,\forall\,x\in X.$ As $J$ is compact and the balls $B(j,\delta)$, $j\in J,$ form an open covering of $J$ we can obtain a finite set $J_0\subset J$ such that the balls $B(j,\delta)$, $j\in J_0$, form an open covering of $J$ too. As this set $J_0$ is finite, there is $n_0\in\mathbb{N}$ such that
	\[ \Gamma(q_j +f\circ\phi_j)-\epsilon <  \frac{1}{\beta}\log(\int_X e^{\beta(q_j+ f\circ\phi_j)}d\rho_{\beta A}) < \Gamma(q_j + f\circ\phi_j)+\epsilon,\,\,\forall j\in J_0, \forall \beta>n_0.\]
	For any $l\in J$ there is $j\in J_0$ such that $d(l,j)<\delta$, then, $q_j+f\circ\phi_j-\epsilon <q_l+f\circ\phi_l<q_j+f\circ\phi_j+\epsilon$. From definition of $\Gamma$ we immediately get $\Gamma(f+a)=\Gamma(f)+a$ for any $a\in\mathbb{R}$ and $[f\leq g \Rightarrow \Gamma(f)\leq \Gamma(g)]$.
	Then, for $\beta>n_0$,
	\[\Gamma(q_l+f\circ\phi_l)-3\epsilon \leq \Gamma(q_j+f\circ\phi_j+\epsilon)-3\epsilon = \Gamma(q_j+f\circ\phi_j)-2\epsilon\]
	\[ < \frac{1}{\beta}\log\left(\int_X e^{\beta( q_j+ f\circ\phi_j)}\,d\rho_{\beta A}\right)-\epsilon = \frac{1}{\beta}\log(\int_Xe^{\beta (q_j+f\circ\phi_j-\epsilon)}\,d\rho_{\beta A})\]
	\[\leq \frac{1}{\beta}\log\left(\int_Xe^{\beta(q_l+ f\circ\phi_l)}\,d\rho_{\beta A}\right) <\frac{1}{\beta}\log\left(\int_Xe^{\beta (q_j+f\circ\phi_j+\epsilon)}\,d\rho_{\beta A}\right)\]
	\[<\Gamma(q_j+f\circ\phi_j)+2\epsilon = \Gamma(q_j+f\circ\phi_j-\epsilon)+3\epsilon \leq \Gamma(q_l+f\circ\phi_l)+3\epsilon.\]
\end{proof}

\begin{proof}[\,\,of Theorem \ref{teo : irreducible}. ]  From Proposition \ref{prop: irreducible unique invariant} there exists a unique invariant idempotent probability $\rho$ for the mpIFS $(X,\phi,q)$.
	Fix  $\tilde{f} \in  C(X,\mathbb{R})$ and suppose by contradiction there exists an increasing sequence of positive numbers $(\beta_i)$ such that $$\lim_{\beta_i\to\infty}\frac{1}{\beta_i}\log (\rho_{\beta_i A}(e^{\beta_i \tilde{f}}))\neq \rho(\tilde{f}).$$
	
	As $C(X,\mathbb{R})$ is separable, there exists a enumerable set $F$  of continuous functions which is dense in $C(X,\mathbb{R})$ with the uniform norm. By applying a Cantor's diagonal argument we can construct a subsequence (also denoted by $(\beta_i)$) such that  the limit $\lim_{\beta_i\to +\infty}\frac{1}{\beta_i}\log(\rho_{\beta_i A}(e^{\beta_i f} ))$ there exists for any $f\in F$.
	We denote by $$\Gamma(f):=\lim_{\beta_i\to\infty}\frac{1}{\beta_i}\log (\rho_{\beta_i A}(e^{\beta_i f}))$$ for any $f\in F$.
	
	\textit{Claim 1.} The limit $\Gamma(g):=\lim_{\beta_i\to\infty}\frac{1}{\beta_i}\log (\rho_{\beta_i A}(e^{\beta_i g}))$ there exists for any $g \in C(X,\mathbb{R})$.
	
	\textit{Proof of Claim 1.} Given $\epsilon >0$ there exists $f\in F$ such that
	$f(x)-\epsilon \leq g(x) \leq f(x)+\epsilon$ for any $x\in X$. Therefore
	\[\Gamma(f)-\epsilon = \lim_{\beta_i\to\infty}\frac{1}{\beta_i}\log (\rho_{\beta_i A}(e^{\beta_i (f-\epsilon)})) \leq \liminf_{\beta_i\to\infty}\frac{1}{\beta_i}\log (\rho_{\beta_i A}(e^{\beta_i g}))\]
	\[\leq \limsup_{\beta_i\to\infty}\frac{1}{\beta_i}\log (\rho_{\beta_i A}(e^{\beta_i g}))\leq \lim_{\beta_i\to\infty}\frac{1}{\beta_i}\log (\rho_{\beta_i A}(e^{\beta_i (f+\epsilon)})) \leq \Gamma(f)+\epsilon.\]
	Then
	\[\limsup_{\beta_i\to\infty}\frac{1}{\beta_i}\log (\rho_{\beta_i A}(e^{\beta_i g})) - \liminf_{\beta_i\to\infty}\frac{1}{\beta_i}\log (\rho_{\beta_i A}(e^{\beta_i g})) \leq 2\epsilon.\]
	Making $\epsilon\to 0$ we get the existence of the limit $\Gamma(g)$. This finish the proof of claim.
	
	\bigskip
	
	\textit{Claim 2.} The map $\Gamma:C(X,\mathbb{R})\to \mathbb{R}$, given by  $\Gamma(f)=\lim_{{\beta_i}\to\infty}\frac{1}{{\beta_i}}\log (\rho_{\beta_i A}(e^{{\beta_i} f}))$ defines an idempotent probability.

	\textit{Proof of Claim 2.}
	\begin{enumerate}
		\item $\Gamma(0)=\lim_{{\beta_i}\to\infty}\frac{1}{{\beta_i}}\log(\rho_{\beta_i A}(1))=0$;
		\item $\Gamma(\lambda \odot \psi)=\lim_{{\beta_i}\to\infty}\frac{1}{{\beta_i}}\log(\rho_{\beta_i A}(e^{{\beta_i} \lambda}\cdot e^{{\beta_i} \psi}))=\lambda \odot \Gamma(\psi)$, for all $\lambda \in\mathbb{R}$ and $\psi \in C(X,\mathbb{R})$;
		\item $\Gamma(\varphi \oplus \psi)= \Gamma(\varphi) \oplus \Gamma(\psi)$, for all $\varphi, \psi \in C(X,\mathbb{R})$. \newline To check this last equation just consider the sequence of inequalities:
		\[\Gamma(\varphi \oplus \psi) = \lim_{{\beta_i}\to\infty}\frac{1}{{\beta_i}}\log(\rho_{\beta_i A}(e^{{\beta_i}(\varphi \oplus \psi)})\geq \lim_{{\beta_i}\to\infty}\frac{1}{{\beta_i}}\log(\rho_{\beta_i A}(e^{{\beta_i}\varphi}) = \Gamma(\varphi).\] Similarly we get such inequality for $\psi$ and so \[\Gamma(\varphi \oplus \psi)\geq \Gamma(\varphi) \oplus \Gamma(\psi).\] On the other hand
		\[\Gamma(\varphi \oplus \psi) = \lim_{{\beta_i}\to\infty}\frac{1}{{\beta_i}}\log(\rho_{\beta_i A}(e^{{\beta_i}(\varphi \oplus \psi)}))\]\[\leq
		\lim_{{\beta_i}\to\infty}\frac{1}{{\beta_i}}\log(\rho_{\beta_i A}(e^{{\beta_i} \varphi})+\rho_{\beta_i A}(e^{{\beta_i} \psi}) ) =  \Gamma(\varphi) \oplus \Gamma(\psi).\]
	\end{enumerate}

	\textit{Claim 3.} The idempotent probability $\Gamma$ is invariant for the operator $\mathcal{M}_{\phi,q}^0$.
	
	\textit{Proof of Claim 3.}  Replace $f$ by $e^{\beta f}$ in equation \eqref{eq: rho invariant} and then take $\lim_{\beta_i\to\infty} \frac{1}{\beta_i}\log(\cdot)$ in both sides of this equation. In this way we get
	\[\frac{1}{\beta_i}\log\left(\int_X e^{\beta_i f}\,d\rho_{\beta_i A}(x)\right) =\frac{1}{\beta_i}\log\left( \int_X\int_J q_j^{\beta_i}(x) \cdot (e^{\beta_i f\circ \phi_j(x)})\,d\nu(j)d\rho_{\beta_i A}(x)\right).\]
	Applying  Fubini's theorem we have
	\[\frac{1}{\beta_i}\log\left(\int_X e^{\beta_i f}\,d\rho_{\beta_i A}(x)\right) =\frac{1}{\beta_i}\log\left( \int_J\int_X q_j^{\beta_i}(x) \cdot (e^{\beta_i f\circ \phi_j(x)})\,d\rho_{\beta_i A}(x)d\nu(j)\right)\]
	\[=		\frac{1}{\beta_i}\log\left( \int_J\int_X e^{\beta_i [\frac{1}{\beta_i}\log\left(q_j^{\beta_i}(x) \cdot (e^{\beta_i f\circ \phi_j(x)})\right)]}\,d\rho_{\beta_i A}(x)d\nu(j)\right).\]
	As $\frac{1}{\beta_i}\log(q_j^{\beta_i}(x) \cdot (e^{\beta_i f\circ \phi_j(x)})) \to q(j,x) + f\circ \phi_j(x)$ uniformly, applying Lemma \ref{lemma :tec1} and Lemma \ref{lemma: tec2} we get that  $\frac{1}{\beta_i}\log\int_X e^{\beta_i [\frac{1}{\beta_i}\log(q_j^{\beta_i}(x) \cdot (e^{\beta_i f\circ \phi_j(x)}))]}\,d\rho_{\beta_i A}(x)$ converges  to $\Gamma(q_j + f\circ \phi_j)$ uniformly on $J$. Then, applying Lemma \ref{lmms}, we get
	\[	\frac{1}{\beta_i}\log\left( \int_J\int_X e^{\beta_i [\frac{1}{\beta_i}\log(q_j^{\beta_i}(x) \cdot (e^{\beta_i f\circ \phi_j(x)}))]}\,d\rho_{\beta_i A}(x)d\nu(j)\right) \to \sup_{j\in J} [\Gamma(q_j + f\circ \phi_j)].\]
	Therefore, 	
	\[\Gamma(f) = \sup_j  [\Gamma(q_j+f\circ \phi_j)]\,\,\,\forall \,f\in C(X,\mathbb{R}).\]
	This shows that $\Gamma$ is invariant for the operator $\mathcal{M}_{\phi,q}^0$
	acting on idempotent probabilities. This proves Claim 3.
	
	\bigskip
	
	From Proposition \ref{prop: irreducible unique invariant},  there exists a unique invariant idempotent  probability for $\mathcal{M}_{\phi,q}^0$ which we denote by $\rho$. Then $\Gamma = \rho$ and the initial hypothesis of this proof is false. We conclude that for any continuous function $f$,  $\rho(f) =\lim_{\beta\to\infty}\frac{1}{\beta}\log (\rho_{_{\beta A}}(e^{\beta f}))$. Let us call by $\lambda$ the density of $\rho$. Then  $\lambda \in U(X,\mathbb{R}_{\max})$ and from Proposition \ref{prop: irreducible unique invariant} it satisfies
	\[\lambda(x) =  S_q(x,z),\,\,\forall z \in \Omega_A.\]
	
	We still need to prove that the Gibbs measures $\rho_{_{\beta A}}$ satisfies a LDP with rate function $I(x) = -\lambda(x)$.
	But this is just a consequence of Theorem \ref{teo : densidade Maslov} and Lemma~\ref{lemma: Bric}.
	
\end{proof}

\subsection{Example: the non-place dependent case}

In this subsection we suppose $X$ and $J$ are compact metric spaces. We consider a uniformly contractible IFS $\mathcal{S}=(X, (\phi_j)_{j \in J})$ and also fixed numbers $A_j, j\in {J}$, such that the function $A:J\to\mathbb{R}$, given by $A(j)=A_j$ is Lipschitz continuous.

For each $\beta>0$, the operator $\mathcal{L}_{\phi,e^{\beta A},\nu}$ has the eigenfunction $h_{\beta A}=1$ associated to the eigenvalue $\lambda_{\beta A} = \int_J e^{\beta A_j}\,d\nu(j)$. Then, we consider the numbers given by $q_j^\beta = \frac{e^{\beta A_j}}{\int_J e^{\beta A_i}d\nu(i)},\,j\in J$. Observe that $\int_J q_j^{\beta}d\nu(j)=1$. We will call $(q_j^\beta)_{j\in J}$ a \textbf{Gibbs density} on $J$ for temperature $\frac{1}{\beta}$.
There exists a unique Borel probability $\rho_{_{\beta A}}$ on $X$ (Gibbs probability) which is  invariant for the dual operator $\mathcal{M}_{\phi,q^{\beta},\nu}$. It satisfies
\begin{equation}\label{eqj1} \rho_{_{\beta A}}(f) = \int_J q_j^\beta \cdot \rho_{_{\beta A}}(f\circ \phi_j)\,d\nu(j) \,\,\forall f\in C(X,\mathbb{R}).
\end{equation}

As $q_j^\beta = \frac{e^{\beta A_j}}{\int_J e^{\beta A_i}d\nu(i)}$ we obtain from Lemma \ref{lmms}
\[q_j:=\lim_{\beta{\to\infty}} \frac{1}{\beta}\log	(q_j^\beta) = A_j - \max_{i\in J}A_i.\]
Then $q_j \leq 0$ and $\bigoplus_{j\in J}q_j = 0$. As $A$ is also continuous, we conclude that  $\mathcal{S}=(X, (\phi_j)_{j\in J}, (q_j)_{j\in J})$ is a max-plus IFS.

From Theorem  \ref{thm:represent idemp invar cte} there exists a unique invariant idempotent  probability $\rho$ for $\mathcal{M}_{\phi,q}^0$. This probability satisfies
\begin{equation}\label{eqj2} \rho(f) = \bigoplus_{j\in J} q_j \odot \rho(f\circ \phi_j) \,\,\forall f\in C(X,\mathbb{R})
\end{equation}
and its density $\lambda \in U(X,\mathbb{R}_{\max})$ satisfies
\[\lambda(x) =  \bigoplus_{(j_1,j_2,j_3,...)\in\pi^{-1}(x)}[q_{j_1}+q_{j_2}+q_{j_3}+...].\]

\begin{theorem}\label{thm: one} Consider the above framework. Then the family of Gibbs probabilities $(\rho_{_{\beta A}})$ satisfy a LDP and its  rate function is given by $$I(x) = -\bigoplus_{(j_1,j_2,j_3,...)\in\pi^{-1}(x)}[q_{j_1}+q_{j_2}+q_{j_3}+...].$$
\end{theorem}

\begin{proof} From Theorem \ref{thm:represent idemp invar cte}, if $x,z \in \Omega_A$ we have $S_q(x,z)=S_q(x,x)=0$. Then $\Omega_A$ is irreducible and therefore the proof is consequence of Theorem~\ref{teo : irreducible}.
\end{proof}

\subsection{The place dependent case}

In this section we will prove the following result.

\begin{theorem}\label{thm: LDP place dependent} Suppose that the family of invariant probabilities $(\rho_{_{\beta A}})_{\beta>0}$ satisfy a LDP and denote by $I$ its deviation function. Suppose that there exists the uniform limit $V = \lim_{\beta \to+\infty} \frac{1}{\beta}\log(h_{\beta A})$. Then, the uniform limit $q = \lim_{\beta \to +\infty} \frac{1}{\beta}\log(q^{\beta})$ there exists and it satisfies $q(j,x) = A(j,x) + V(\phi_j( x)) - V(x) -m(A),$ which defines a normalized family of weights. Furthermore, $-I$ is the density of an idempotent invariant probability $\rho$ for the mpIFS $\mathcal{S}=(X, \phi, q)$, that is
\begin{equation}\label{eq : -I}
	 -I(x) =\sup_{z\in \Omega} [S_q(x,z) -I(z)]  ,\,\,\forall x\in X .
\end{equation}
\end{theorem}

\begin{proof}	Suppose that the family of invariant probabilities $(\rho_{_{\beta A}})_{\beta>0}$ satisfy a LDP and denote by $I$ its deviation function. Then $I$ is l.s.c and $\inf_{x\in X}I(x) = 0$. By applying Lemma \ref{lemma: Varadhan} we get that, for any $f\in C(X,\mathbb{R})$, there exists the limit $\Gamma(f) := \lim_{\beta\to+\infty}\frac{1}{\beta}\log(\rho_{_{\beta A}}(e^{\beta f}))$ and it satisfies
	\[\Gamma(f) =  \sup_{x\in X}[f(x)-I(x)].\]	
	Therefore, $\Gamma$ defines an idempotent probability with density $-I$ (observe that $-I$ is u.s.c and $\sup_{x\in X} -I(x) =0$).

	Suppose there exists the uniform limit $V = \lim_{\beta \to+\infty} \frac{1}{\beta}\log(h_{\beta A})$. Then $V$ is a calibrated subaction to $A$.
	We have, by definition of $q^{\beta}$,
	\[\frac{1}{\beta}\log (q^{\beta} (j,x)) = A(j,x) +\frac{1}{\beta}\log(h_{\beta A}(\phi_j(x))) -  \frac{1}{\beta}\log(h_{\beta A}(x)) -\frac{1}{\beta}\log(\lambda_{\beta A}).\]
	Then  $\frac{1}{\beta}\log (q^{\beta} (j,x)) \to q(j,x)$ uniformly, where $q(j,x) = A(j,x) + V(\phi_j( x)) - V(x) -m(A)$.
	As $V$ is a calibrated subaction, the function  $q$ defines a normalized family of weights. Let us consider the  mpIFS $\mathcal{S}=(X, \phi, q)$.
	
	\bigskip
	
	\noindent
	\textit{Claim:} $\Gamma$ defines an idempotent probability which is invariant for the operator $\mathcal{M}_{\phi,q}^0$.
		
		\noindent
	\textit{Proof of Claim:} It follows from the same reasoning as Claim 3 in the proof of Theorem \ref{teo : irreducible}.
	
	\bigskip

Finally, as $-I$ is the density of an  invariant idempotent probability for the mpIFS $\mathcal{S}=(X, \phi, q)$, we obtain, from Theorem \ref{teo: main result} that
	\[-I(x) =\sup_{z\in \Omega} [S_q(x,z) -I(z)]. \]
\end{proof} 	

Is a remarkable fact that equations \eqref{eq : -I} and \eqref{eq : V Mane} represents a kind of duality between calibrated subactions (associated to eigenfunction in zero temperature) and deviation functions (associated to Gibbs probabilities in zero temperature). On the other hand LDP provides a connection between Gibbs probabilities in thermodynamic formalism and invariant idempotent probabilities in mpIFS.

\end{document}